\documentclass{article}
\usepackage{amsmath, amsthm, amssymb, mathtools, leftindex, tikz-cd, verbatim, bbm}
\usepackage{enumitem}
\usepackage[parfill]{parskip}
\usepackage[a4paper, total={7in, 9in}]{geometry}
\usepackage[hidelinks]{hyperref}

\usepackage{pgfplots}
 
\pgfplotsset{compat = newest}

\usepackage[maxbibnames=99, backend= biber, style = alphabetic]{biblatex}
\mathtoolsset{showonlyrefs}

\newtheorem{thm}{Theorem}[section]
\newtheorem{lem}[thm]{Lemma}
\newtheorem{prop}[thm]{Proposition}

\newtheorem{defi}[thm]{Definition}
\newtheorem{cor}[thm]{Corollary}

\newtheoremstyle{named}{}{}{\itshape}{}{\bfseries}{.}{.5em}{\thmnote{#3}}

\theoremstyle{named}

\newcommand{\D}{\mathcal{D}}

\newcommand{\B}{\mathcal{B}}

\newcommand{\Z}{\mathbb{Z}}

\newcommand{\Q}{\mathbb{Q}}
\renewcommand{\S}{\mathcal{S}}

\newcommand{\q}{\mathfrak{q}}

\newcommand{\CC}{\mathbb{C}}
\newcommand{\Hyp}{\mathbb{H}}
\newcommand{\R}{\mathbb{R}}

\newcommand{\N}{\mathcal{N}}
\renewcommand{\O}{\mathcal{O}}

\DeclareMathOperator{\disc}{Disc}

\DeclareMathOperator{\PSL}{PSL}
\DeclareMathOperator{\SL}{SL}

\DeclareMathOperator{\vol}{Vol}

\DeclareMathOperator{\Tr}{Tr}
\DeclareMathOperator{\modulo}{mod}

\begin{document}
\title{On the prime geodesic theorem}
\author{By \textit{Alberto Acosta Reche} at University College London}
\date{}
\maketitle

\begin{abstract}
We improve the error term in the prime geodesic theorem for the modular group. More generally, our result applies to any congruence subgroup of any indefinite quaternion algebra over $\Q$. 
\end{abstract}

\section{Introduction.}
The prime geodesic theorem for the modular group $\Gamma = \PSL_2(\Z)$ is concerned with the asymptotics of
\begin{equation}\label{eq:introdefichebyshev}
    \Psi_\Gamma(x) := \sum_{\{P\}_\Gamma: N(P) \leq x} \log N(P_0)
\end{equation}
as $x \to \infty$, where the sum ranges over hyperbolic $\Gamma$-conjugacy classes, $P_0$ is the primitive element corresponding to $P$, and $N(P)$ is the norm of $P$. An alternative expression for $\Psi_\Gamma(x)$ which better captures its number-theoretic content is
\begin{equation}\label{eq:connectiontoLvalues}
    \Psi_{\Gamma}(x) = 2\sum_{2 < n \leq X} \sum_{du^2 = n^2 - 4}  \sqrt{d}L(1; \chi_d),
\end{equation}
where $X = \sqrt{x} + 1/\sqrt{x}$, and $\chi_d$ is the quadratic character (not necessarily primitive) associated to the discriminant $d$, and where $L(s, \chi_d)$ is the corresponding Dirichlet series, see \cite[Equation (1)]{Soundararajan_2013}. Using the Selberg trace formula one can give a quick proof of the asymptotic
\begin{equation}\label{eq:baselinebound}
    \Psi_\Gamma(x) = x + O(x^{3/4}),
\end{equation}
see \cite[Theorem 10.5]{iwaniec}. The same asymptotic holds for an arbitrary cofinite Fuchsian group, except for the possible presence of an explicit secondary term. Several authors were able to improve the error term in \eqref{eq:baselinebound} by using the Bruggeman--Kuznetsov trace formula together with deep number-theoretic inputs, see the works of Iwaniec \cite{Iwaniec1984}, Luo and Sarnak \cite{luosarnak} and Cai \cite{Cai2002}. Using identity \eqref{eq:connectiontoLvalues} Bykovskii was able to prove the local asymptotic
\begin{equation}\label{eq:bykovskiilocalresult}
    \Psi_\Gamma(x + u) - \Psi_\Gamma(x) = u + o_\varepsilon(u),
\end{equation}
where $x^{1/2 + \varepsilon} \leq u \leq x$, see \cite[Theorem 1]{bykovskiiprimegeodesicshortintervals}. Combining ideas from \cite{bykovskiiprimegeodesicshortintervals} with results from \cite{luosarnak},  Soundararajan and Young \cite{Soundararajan_2013} proved that
\begin{equation}\label{eq:soundyoungconditional}
    \Psi_\Gamma(x) = x + O_\varepsilon(x^{2/3 + \theta/6 + \varepsilon}),
\end{equation}
where $\theta > 0$ is any subconvexity exponent in level aspect for quadratic Dirichlet $L$-functions. By work of Conrey and Iwaniec \cite{conreyiwaniec}, Soundararajan and Young were allowed to use $\theta = 1/6 + \varepsilon$ unconditionally in \eqref{eq:soundyoungconditional}, thus arriving at
\begin{equation}\label{eq:soundyoungunconditional}
    \Psi_\Gamma(x) = x + O_\varepsilon(x^{25/36 + \varepsilon}),
\end{equation}
which remains to this day the best unconditional asymptotic for $\Psi_\Gamma(x)$. In this paper we will improve \eqref{eq:soundyoungunconditional} unconditionally, see Corollary \ref{cor:pgtforcongruencesubgroups} below. Before stating and proving a more general result, we would first like to convince the reader that the error term in \eqref{eq:soundyoungunconditional} can indeed be improved.

\subsection{Main idea.}\label{subsec:mainidea}
Let $\omega \in C_c^\infty(0, 1)$ satisfy $\int_\R \omega(u)\, du = 1$, let $Y$ be a parameter satisfying $x^{1/2 + \varepsilon} \leq Y \leq x/\log x$ and define $\omega_Y(u) := \omega(u/Y)/Y$. Using previous work of Luo and Sarnak \cite[Equation (58)]{luosarnak}, Soundararajan and Young evaluated a smoothed version of $\Psi_\Gamma(x)$, see \cite[Equation (13) in Theorem 3.1]{Soundararajan_2013}. After a small rearrangement, their result can be written as
\begin{equation}\label{eq:soundyoungsmoothing}
    \Psi_\Gamma(x) = x - \int_0^Y (\Psi_\Gamma(x + u) - \Psi_\Gamma(x) - u) \omega_Y(u)\, du + O_\varepsilon(x^{7/8 + \varepsilon}Y^{-1/4}).
\end{equation}
On the other hand, Theorem 3.2 in \cite{Soundararajan_2013} asserts that 
\begin{equation}\label{eq:soundyounglocalresult}
    \Psi_\Gamma(x + u) - \Psi_\Gamma(x) =  u + O_\varepsilon(u^{1/2}x^{1/4 + \theta/2 + \varepsilon})
\end{equation}
for $x^{1/2 + \varepsilon} \leq u \leq x/\log x$. Combining \eqref{eq:soundyoungsmoothing} and \eqref{eq:soundyounglocalresult} and choosing the optimal value $Y = x^{5/6 - 2\theta/3}$ one arrives at \eqref{eq:soundyoungconditional}. Thus, an improvement to \eqref{eq:soundyoungunconditional} would follow if we could improve \eqref{eq:soundyounglocalresult} in the range $u \approx x^{13/18}$, with $\theta = 1/6$. 

A key step in the proof of both \eqref{eq:bykovskiilocalresult} and \eqref{eq:soundyounglocalresult} is the identity 
\begin{equation}\label{eq:localexpressionsoundyoung}
    \Psi_\Gamma(x + u) - \Psi_\Gamma(x) = u - \sum_{X < n \leq X'}2\sqrt{n^2 - 4}I_V(n^2 - 4) + O_\varepsilon(uV^{-1/2} + x^{1/2}V^{1/2 + \varepsilon}),
\end{equation}
which can be found implicitly in both \cite[Section 2]{bykovskiiprimegeodesicshortintervals} and \cite[Section 5]{Soundararajan_2013}. In formula \eqref{eq:localexpressionsoundyoung}, the parameter $V$ satisfies $1 \leq V \leq x$ and the quantities $X$ and $X'$ are defined as
\begin{equation*}
X := \sqrt{x} + 1/\sqrt{x},\quad X' := \sqrt{x + u} + 1/\sqrt{x + u},
\end{equation*}
so that, in particular, we have $|X' - X| \asymp ux^{-1/2} \approx x^{2/9}$. Furthermore, $I_V(n^2 - 4)$ is a contour integral given by
\begin{equation*}
    I_V(n^2 - 4) := \frac{1}{2\pi i}\int_{(\sigma)}L(s, n^2 - 4)V^{s - 1}\Gamma(s - 1)\, ds
\end{equation*}
with $0< \sigma < 1$, where we have
\begin{equation}\label{eq:bykovskiifunction}
    L(s, n^2 - 4) := \sum_{du^2 = n^2 - 4}L(s, \chi_d)u^{1-2s}
\end{equation}
with $\chi_d$ being the (not necessarily primitive) real character associated to the discriminant $d$. In their proof of the local asymptotic \eqref{eq:soundyounglocalresult}, Soundarajan and Young estimate each $I_V(n^2-4)$ by letting $\sigma = 1/2$ and using subconvexity in level aspect to obtain
\begin{equation}\label{eq:soundyoungboundsubconvexity}
    I_V(n^2 - 4) \ll x^{\theta + \varepsilon}V^{-1/2}.
\end{equation}
Summing over $n$ they arrive at
\begin{equation}\label{eq:soundyoungsumintegrals}
    \sum_{X < n \leq X'}2\sqrt{n^2 - 4}I_V(n^2 - 4) \ll ux^{\theta + \varepsilon}V^{-1/2},
\end{equation}
which combined with \eqref{eq:localexpressionsoundyoung} leads to \eqref{eq:soundyounglocalresult} if we take $V = ux^{-1/2 + \theta}$. Thus, in order to improve on \eqref{eq:soundyoungunconditional} it is enough to improve on \eqref{eq:soundyoungsumintegrals} in the range $u \approx x^{13/18}$, $V \approx x^{7/18}$, with $\theta = 1/6$.

Motivated by the proof of \eqref{eq:bykovskiilocalresult} in \cite{bykovskiiprimegeodesicshortintervals}, we introduce the set 
\begin{equation*}
    \B_\sigma := \{n \in (X, X'] : L(s, n^2 - 4) \text{ has a zero in }[\sigma, 1] \times [-x^\varepsilon, x^\varepsilon]\},
\end{equation*} 
where $\sigma$ satisfies $1/2 \leq \sigma \leq 1$ and is to be chosen at the end of the argument. If $n \in (X, X'] \setminus \B_\sigma$, then a standard argument \cite[Lemma 1]{bykovskiiprimegeodesicshortintervals} shows that $L(s, n^2 - 4)$ satisfies the Lindelöf hypothesis on $[\sigma + \varepsilon, 1]\times [-x^\varepsilon + 1, x^\varepsilon - 1]$. It follows that for $n \in (X, X'] \setminus \B_\sigma$ we have 
\begin{equation}\label{eq:boundgoodfunctions}
    I_V(n^2 - 4) \ll x^\varepsilon V^{\sigma - 1}.
\end{equation}
If we use bound \eqref{eq:soundyoungboundsubconvexity} for $n \in \B_\sigma$ and bound \eqref{eq:boundgoodfunctions} for $n \in (X, X'] \setminus \B_\sigma$ we arrive at 
\begin{equation}\label{eq:boundsumintegrals}
    \sum_{X < n \leq X'}2 \sqrt{n^2 - 4}I_V(n^2 - 4) \ll |\B_\sigma|x^{1/2 + \theta + \varepsilon}V^{-1/2} + ux^\varepsilon V^{\sigma - 1}.
\end{equation}
Heath-Brown's zero-density estimate for quadratic characters \cite[Theorem 3]{heathbrownquadlargesieve} implies that
\begin{equation}\label{eq:introheathbrownestimate}
    |\B_\sigma| \ll_\varepsilon x^{3(1-\sigma)/(2-\sigma)+\varepsilon}.
\end{equation}
Recall that $\theta = 1/6$, $u \approx x^{13/18}$ and $V \approx x^{7/18}$. If we insert \eqref{eq:introheathbrownestimate} in \eqref{eq:boundsumintegrals} and let $\sigma \in (23/25, 13/14)$ we can improve the bound in \eqref{eq:soundyoungsumintegrals} which, by our discussion, leads to a better bound for the error term in \eqref{eq:soundyoungunconditional}.

Our argument can be refined a bit more by looking into the proof of \eqref{eq:introheathbrownestimate}. Indeed, by inspecting \cite[Section 11]{heathbrownquadlargesieve} we observe that our application of \cite[Theorem 3]{heathbrownquadlargesieve} is using implicitly the estimate
\begin{equation}\label{eq:introfourthmomentbound}
    \sum_{n \in \B_\sigma} |L(1/2 + it, n^2 - 4)|^4 \ll \sum_{\chi \in S(x)} x^\varepsilon |L(1/2 + it, \chi)|^4 \ll x^{1 + \varepsilon},
\end{equation} 
where $S(x)$ is the set of primitive quadratic characters of discriminant at most $x$ and $|t| \ll x^\varepsilon$. Since $|\B_\sigma| \leq |X' - X| \asymp ux^{-1/2}\approx x^{2/9}$ is very small compared to $|S(x)| \asymp x$, one can hope to improve on \eqref{eq:introfourthmomentbound} by using subconvexity. This is indeed the case, since we have
\begin{equation}\label{eq:introsubconvexbound}
    \sum_{n \in \B_\sigma} |L(1/2 + it, n^2 - 4)|^4 \ll |\B_\sigma| x^{4\theta + \varepsilon},
\end{equation} 
which is better than \eqref{eq:introfourthmomentbound} whenever $|\B_\sigma| \leq x^{1/3}$ (recall that $\theta = 1/6$). Motivated by this observation we will prove a zero-density estimate for quadratic characters which improves on \cite[Theorem 3]{heathbrownquadlargesieve} near $\sigma = 1$, see Proposition \ref{prop:myzerodensity} below.

\subsection{Main results.}\label{subsec:results}
We start with a definition which captures the subconvexity bounds that will feature in the results below.
\begin{defi}
If $\psi$ is a Dirichlet character, we say that a pair $(\theta, A)\in [0, 1/4]\times [0, \infty)$ is $\psi$-\emph{admissible} if for all $t \in \R$ and all primitive quadratic characters $\chi$ we have 
\begin{equation}\label{eq:defipsiadmissible}
    |L(1/2 + it, \chi \psi)| \ll_{\psi, \varepsilon} (1 + |t|)^{A + \varepsilon} \q(\chi)^{\theta + \varepsilon},
\end{equation}
where $\q(\chi)$ is the conductor of $\chi$. We say that $\theta \in [0, 1/4]$ is admissible if for any $\psi$ there is some $A \geq 0$, possibly depending on $\psi$, such that the pair $(\theta, A)$ is $\psi$-admissible. 
\end{defi}
By work of Petrow and Young we know that the pair $(\theta, A) = (1/6, 1/6)$ is $\psi$-admissible for any $\psi$, see \cite{petrowyoung20} and \cite{petrowyoung}. In particular, the exponent $\theta = 1/6$ is admissible.

In order to state a general result, we need to recall the setup of quaternion algebras from \cite{reche2026chebotarevgeodesictheoremnonsplit}. We let $B$ be an indefinite quaternion algebra over $\Q$ and we let $\O$ be a maximal order in $B$. Using the isomorphism $B \otimes_\Q \R \simeq M_2(\R)$ we can view $\O^1 \subset \SL_2(\R)$ as a discrete cofinite subgroup. We let $q \geq 1$ be a positive integer and denote by $\O^1(q) := \O^1 \cap (1 + q\O)$ the principal congruence subgroup of $\O^1$ of level $q$. We let $\Gamma$ be a congruence subgroup $\O^1(q) \subset \Gamma \subset \O^1$, and define $\widetilde{\Gamma} := \O^1(q)\backslash \Gamma$. Finally, we let $\varphi$ be a class function on $\widetilde{\Gamma}$ and define

\begin{equation}\label{eq:generalchebyshevsum}
\Psi_{\Gamma}(x; \varphi) := [\{\pm I\}\Gamma: \Gamma]\sum_{\substack{\{P\}_{\Gamma}\\N(P) \leq x\\\Tr(P) > 2}} \log N(P_0)\varphi(P \modulo q).
\end{equation} 

The reader only interested in the case of $\Gamma = \SL_2(\Z)$ may take $B = M_2(\Q)$, $\O = M_2(\Z)$, $q = 1$ and $\varphi = \mathbbm{1}$, in which case \eqref{eq:generalchebyshevsum} agrees with \eqref{eq:introdefichebyshev}. The main result of this paper is a strong asymptotic for the sum $\Psi_\Gamma(x; \varphi)$ as $x \to \infty$.

\begin{thm}\label{thm:maintheorem}
For any $q \geq 1$, for any class function $\varphi$ on $\widetilde{\Gamma}$, and for any admissible exponent $\theta$ we have 
\begin{equation*}
    \Psi_{\Gamma}(x; \varphi) =  \langle \varphi, \mathbbm{1}\rangle x + O_{\varphi, \varepsilon}(x^{\eta(\theta) + \varepsilon}),
\end{equation*} 
where 
\begin{equation*}
    \langle \varphi, \mathbbm{1}\rangle := \frac{1}{|\widetilde{\Gamma}|}\sum_{\gamma \in \widetilde{\Gamma}}\varphi(\gamma)
\end{equation*}
and
\begin{equation}\label{eq:definitionexponent}
    \eta(y) := 1/2 + \frac{\sqrt{81y^2 + 106y + 33}-5 - 9y}{4}.
\end{equation}
\end{thm}

For any $\Gamma$ as above we let $\Psi_\Gamma(x) := \Psi_\Gamma(x; \mathbbm{1})$, so that 
\begin{equation*}
    \Psi_\Gamma(x) = [\{\pm I\}\Gamma: \Gamma]\sum_{\substack{\{P\}_\Gamma\\
    N(P)\leq x\\
    \Tr(P) > 2}}\log N(P_0).
\end{equation*} 
Thus, as a particular case of Theorem \ref{thm:maintheorem} we deduce the following corollary. 

\begin{cor}\label{cor:pgtforcongruencesubgroups}
Let $\Gamma\subset \SL_2(\R)$ be a congruence subgroup of an indefinite quaternion algebra over $\Q$. Then, for any admissible exponent $\theta$ we have
\begin{equation*}
    \Psi_\Gamma(x) = x + O_\varepsilon(x^{\eta(\theta) + \varepsilon}).
\end{equation*}
\end{cor}

In particular, since $\theta = 1/6$ is admissible and  
\begin{equation*}
    \eta(1/6) \approx 0.69359 < 0.69\overline{4} = \frac{25}{36}
\end{equation*}
we see that Corollary \ref{cor:pgtforcongruencesubgroups} improves unconditionally on \cite[Theorem 1.1]{Soundararajan_2013}.
\subsection{Acknowledgements.}
The author would like to express his gratitude to his advisor Ian Petrow for many helpful discussions and for carefully reading the present work. The author would like to thank Abdellah Sebbar for pointing out typos. This work was supported by the Engineering and Physical Sciences Research Council [EP/S021590/1], via the EPSRC Centre for Doctoral Training in Geometry and Number Theory (The London School of Geometry and Number
Theory), University College London.

\section{A zero-density estimate.}\label{sec:zerodensity}

Let $Q$ be a positive integer, and let $S(Q)$ be the set of primitive real characters of conductor up to $Q$. Heath-Brown showed that this family of arithmetic functions has strong orthogonality properties. Indeed, he proved \cite[Corollary 3]{heathbrownquadlargesieve} that the inequality
\begin{equation}\label{eq:quadlargesieve}
    \sum_{\chi \in S(Q)}\left|\sum_{n \leq N} a_n \chi(n)\right|^2 \ll_\varepsilon (QN)^\varepsilon (Q + N)N\max_{n \leq N}|a_n|^2
\end{equation}
holds for an arbitrary sequence $(a_n)$. A consequence of this large sieve inequality is the fourth-moment bound 
\begin{equation}\label{eq:fourthmomentbound}
    \sum_{\chi \in S(Q)} |L(1/2 + it, \chi)|^4 \ll (Q(|t| + 1))^{1 + \varepsilon},
\end{equation} 
see \cite[Theorem 2]{heathbrownquadlargesieve}. Let $N(\sigma, T, \chi)$ be the number of zeroes of $L(s, \chi)$ in the rectangle 
\begin{equation*}
    \sigma \leq \text{Re}(s) \leq 1, |\text{Im}(s)| \leq T,
\end{equation*}
where $1/2\leq\sigma\leq 1$ and $T\geq 1$.
Using \eqref{eq:quadlargesieve} as well as \eqref{eq:fourthmomentbound}, Heath-Brown was able to prove the zero-density estimate
\begin{equation}\label{eq:zerodensityheathbrown}
    \sum_{\chi \in S(Q)} N(\sigma, T, \chi) \ll_\varepsilon (QT)^\varepsilon Q^{3(1-\sigma)/(2-\sigma)}T^{(3-2\sigma)/(2-\sigma)},
\end{equation}
see \cite[Theorem 3]{heathbrownquadlargesieve}. Motivated by the discussion in Section \ref{subsec:mainidea}, we will use a subconvexity bound instead of \eqref{eq:fourthmomentbound} to arrive at the following zero-density estimate.
\begin{prop}\label{prop:myzerodensity}
For any Dirichlet character $\psi$ and any $\psi$-admissible pair $(\theta, A)$ we have
\begin{equation*}
    \sum_{\chi \in S(Q)} N(\sigma, T, \chi\psi) \ll_{\psi, \varepsilon} (QT)^\varepsilon Q^{2(1 + 2\theta)(1-\sigma)}T^{1 + 4A(1-\sigma)}.
\end{equation*}
\end{prop}
By work of Petrow and Young we can take $(\theta, A) = (1/6, 1/6)$ in Proposition \ref{prop:myzerodensity}, see \cite{petrowyoung20} and \cite{petrowyoung}. In our application to Theorem \ref{thm:maintheorem} we will only need to use Proposition \ref{prop:myzerodensity} in $Q$-aspect. Note that Proposition \ref{prop:myzerodensity} improves on \eqref{eq:zerodensityheathbrown} in $Q$-aspect when $\theta = 1/6$ and  $\sigma \geq 7/8$. 

\begin{proof}[Proof of Proposition \ref{prop:myzerodensity}]
In this proof the character $\chi\psi$ is given by $n \mapsto \chi(n)\psi(n)$ and is not necessarily primitive. This is not an issue since $L(s, \chi\psi)$ and $L(s, (\chi\psi)^*)$ have the same zeros in the region $\text{Re}(s) > 0$, where $(\chi\psi)^*$ is the primitive character corresponding to $\chi\psi$. To prove Proposition \ref{prop:myzerodensity} we will follow \cite[Section 11]{heathbrownquadlargesieve} very closely. First of all, note that we can assume that $\sigma \geq 1/2 + \theta/(1 + 2\theta) $, since otherwise Proposition \ref{prop:myzerodensity} follows from the Riemann--von Mangoldt formula \cite[Theorem 5.8]{iwanieckowalski}. Also, since $\psi$ is fixed, we can assume that $\chi\psi$ is nontrivial by ignoring the contribution of at most a single $\chi$.  Define
\begin{equation*}
\begin{aligned}
    \mathcal{R}(\sigma, \tau) & := [\sigma, \sigma + (\log QT)^{-1}] \times [\tau, \tau + (\log QT)^{-1}],\\
    N'(\sigma, \tau) & := \#\{\chi \in S(Q): L(s, \chi \psi) \text{ has a zero in } \mathcal{R}(\sigma, \tau)\}.
\end{aligned}
\end{equation*}
By subdividing the region $[\sigma, 1]\times [-T, T]$ into small rectangles and summing the different contributions trivially, we see that in order to prove Proposition \ref{prop:myzerodensity} it is enough to show that  
\begin{equation}\label{eq:tobeproved}
    N'(\sigma, \tau) \ll_{\psi, \varepsilon} (QT)^\varepsilon Q^{2(1 + 2\theta)(1-\sigma)}T^{4A(1-\sigma)}
\end{equation}
holds for all $(\sigma, \tau) \in [1/2 + \theta/(1 + 2\theta), 1]\times [-T, T]$. Let $100 \leq X \leq Y \leq (QT)^{100}$ be large parameters and consider 
\begin{equation}\label{eq:inversedirpoly}
    M_X(s,\chi\psi) := \sum_{n \leq X}\frac{\mu(n)\chi(n)\psi(n)}{n^s}.
\end{equation}
If $\rho = \beta + i\gamma \in \mathcal{R}(\sigma, \tau) $ is a zero of $L(s, \chi\psi)$, then a routine argument \cite[Equation (12.24)]{montgomerytopics} using Mellin inversion shows that 
\begin{equation}\label{eq:auxeq}
    e^{-1/Y} + \sum_{n > X}\frac{a(n)\chi(n)\psi(n)}{n^\rho}e^{-n/Y} = \frac{1}{2\pi i}\int_{(1/2 - \text{Re}(\rho))} L(s + \rho, \chi\psi)M_X(s + \rho, \chi\psi)\Gamma(s)Y^s \, ds,
\end{equation}
where 
\begin{equation*}
    a(n) := \sum_{\substack{d \mid n\\ d \leq X}} \mu(d), \quad \text{ so that }\quad |a(n)| \leq \tau(n).
\end{equation*}
Let $\mathcal{U} := \{2^i X: 0 \leq i \leq \log_2(Y^2/X)\}$. Note that $|\mathcal{U}| \asymp \log Y$. Equation \eqref{eq:auxeq} implies that either we have
\begin{equation*}
\left|\int_{-(QT)^\varepsilon}^{(QT)^\varepsilon} L(1/2 + i\gamma + iu, \chi \psi)M_X(1/2 + i\gamma + iu, \chi \psi)Y^{1/2 - \beta + iu}\Gamma(1/2 - \beta + iu)\, du\right| \gg 1,
\end{equation*} 
or else there is $U \in \mathcal{U}$ such that 
\begin{equation*}
    \left|\sum_{U< n \leq 2U}\frac{a(n)\chi(n)\psi(n)}{n^\rho}e^{-n/Y}\right| \gg |\mathcal{U}|^{-1}
\end{equation*}
Therefore, using Equation \eqref{eq:auxeq} for each $\chi \in S(Q)$ such that $L(s, \chi \psi)$ has a zero in $\mathcal{R}(\sigma, \tau)$ we arrive at two different cases:
\begin{enumerate}
    \item [(1)] There are $\gg N'(\sigma, \tau)$ characters $\chi \in S(Q)$ for which
\begin{equation}\label{eq:firstcase}
    \left|\int_{-(QT)^\varepsilon}^{(QT)^\varepsilon} L(1/2 + i\gamma_\chi + iu, \chi \psi)M_X(1/2 + i\gamma_\chi + iu, \chi \psi)Y^{1/2 - \beta_\chi + iu}\Gamma(1/2 - \beta_\chi + iu)\, du\right| \gg 1
\end{equation}
holds for some $\rho_\chi = \beta_\chi + i\gamma_\chi \in \mathcal{R}(\sigma, \tau)$.
    \item [(2)] Otherwise, there is an integer $U \in \mathcal{U}$ and $\gg N'(\sigma, \tau)|\mathcal{U}|^{-1}$ characters $\chi \in S(Q)$ for which
\begin{equation*}
    \left|\sum_{U< n \leq 2U}\frac{a(n)\chi(n)\psi(n)}{n^{\rho_\chi}}e^{-n/Y}\right| \gg |\mathcal{U}|^{-1}
\end{equation*}
holds for some $\rho_\chi = \beta_\chi + i\gamma_\chi \in \mathcal{R}(\sigma, \tau)$.
\end{enumerate}
We can deal with the second case exactly as in Heath-Brown, obtaining  
\begin{equation*}
    N'(\sigma, \tau) \ll_{\psi, \varepsilon} (QT)^\varepsilon (QX^{1-2\sigma} + Y^{2-2\sigma})
\end{equation*}
in that case. Our proof differs from \cite[Section 11]{heathbrownquadlargesieve} in the treatment of the first case. Squaring \eqref{eq:firstcase}, applying Cauchy-Schwarz and the subconvexity bound from \eqref{eq:defipsiadmissible} we deduce that 
\begin{equation}\label{eq:auxcase2intermediate}
    N'(\sigma, \tau) \ll_{\psi, \varepsilon} (QT)^{\varepsilon} Q^{2\theta}T^{2A} Y^{1-2\sigma}\int_{\tau -(QT)^\varepsilon}^{\tau + (QT)^\varepsilon} \sum_{\chi \in S(Q)}|M_X(1/2 + iu, \chi\psi)|^2 \, du.
\end{equation}
We can write 
\begin{equation*}
    M_X(s, \chi\psi) = \sum_{i = 0}^{\lfloor \log_2 X\rfloor} M_{X, i}(s, \chi\psi)
\end{equation*}
where 
\begin{equation*}
    M_{X, i}(s, \chi\psi) := \begin{dcases}
        \sum_{2^i \leq n < 2^{i + 1}}\frac{\mu(n)\chi(n)\psi(n)}{n^s}, & 0 \leq i < \lfloor \log_2 X\rfloor,\\
        \sum_{2^i \leq n \leq X}\frac{\mu(n)\chi(n)\psi(n)}{n^s}, & i = \lfloor \log_2 X\rfloor.
    \end{dcases}
\end{equation*}
It follows from \eqref{eq:auxcase2intermediate} that 
\begin{equation*}
    N'(\sigma, \tau) \ll_{\psi, \varepsilon} (QT)^{\varepsilon} Q^{2\theta}T^{2A} Y^{1-2\sigma}\sum_{i = 0}^{\lfloor \log_2 X\rfloor}\int_{\tau -(QT)^\varepsilon}^{\tau + (QT)^\varepsilon} \sum_{\chi \in S(Q)}|M_{X, i}(1/2 + iu, \chi\psi)|^2 \, du,
\end{equation*}
and applying the large sieve inequality \eqref{eq:quadlargesieve} we deduce that
\begin{equation*}
    N'(\sigma, \tau) \ll_{\psi, \varepsilon} (QT)^{\varepsilon} Q^{2\theta}T^{2A} (Q + X)Y^{1 - 2\sigma}.
\end{equation*}
It follows that, in any case, we have 
\begin{equation*}
    N'(\sigma, \tau)  \ll_{\psi, \varepsilon} (QT)^\varepsilon \left(Q^{2\theta}T^{2A}(Q + X)Y^{1-2\sigma} + QX^{1-2\sigma} + Y^{2-2\sigma}\right). 
\end{equation*}
If we let $X = Q$ and $Y = Q^{1 + 2\theta}T^{2A}$ we arrive at \eqref{eq:tobeproved}, which concludes the proof of Proposition \ref{prop:myzerodensity}.

\end{proof}

\section{Preparations for the proof of Theorem \ref{thm:maintheorem}.}
Recall from \cite[Definition 1.6]{reche2026chebotarevgeodesictheoremsplit} that a congruence character of level $q$ is a pair $(\Gamma, \chi)$ where $\{\pm I\}\Gamma(q)\subset \Gamma \subset \SL_2(\Z)$ is a congruence subgroup and $\chi: \Gamma\rightarrow S^1$ is a one-dimensional character which is trivial on $\Gamma(q)$. We can also view $\chi$ as a class function on $\widetilde{\Gamma} = \Gamma(q)\backslash \Gamma$. We denote by $k \in \{0, 1\}$ the weight of $\chi$, so that $\chi(-I) = (-1)^k$.
We let $\{u_j\}_j$ be an orthonormal basis of $\chi$-automorphic Maass cuspidal forms with Laplacian eigenvalue $\lambda_j = 1/4 + t_j^2$, see \cite[Section 3.1]{reche2026chebotarevgeodesictheoremsplit} for more details. We know that
\begin{enumerate}
    \item [i)] If $k = 1$, then $\lambda_j \geq 1/4$ \cite[Equation (4.40)]{DFIsubconvexityartin}. In this case we  assume that each $t_j$ satisfies $t_j \geq 0$.  
    \item [ii)] If $k = 0$, then $\lambda_j \geq 1/4 - (7/64)^2$, see the second appendix to \cite{kimramakrishnansarnak2003}. If $t_j$ is purely imaginary, we assume that $\text{Im}(t_j) > 0$, and if $t_j$ is real we assume that $t_j \geq 0$. 
\end{enumerate}
Given $100 \leq T \leq x^{1/2}$ we let 
\begin{equation*}
    \S(T, x) := \sum_{0 \leq t_j \leq T} x^{it_j}.
\end{equation*}   
We know that \cite[Theorem 1.7]{reche2026chebotarevgeodesictheoremsplit}
\begin{equation}\label{eq:luosarnakbound}
     \S(T, x) \ll_\varepsilon T^{5/4}x^{1/8 + \varepsilon},
\end{equation}
which generalizes Equation (58) in \cite{luosarnak}. We also know that  \cite[Lemma 5.2]{reche2026chebotarevgeodesictheoremsplit}
\begin{equation}\label{eq:iwaniecslemma}
    \Psi_\Gamma(x; \chi) = \delta_{\chi = 1}x + \sum_{0 \leq t_j \leq T}\left(\frac{x^{1/2 + it_j}}{1/2 + it_j} + \frac{x^{1/2 - it_j}}{1/2 - it_j}\right) + O\left(\frac{x (\log x)^2}{T} + \delta_{k = 0} x^{1/2 + 7/64}\right)
\end{equation}
for $1 \leq T \leq x^{1/2}(\log x)^{-2}$, which  generalizes \cite[Lemma 1]{Iwaniec1984}. Note that in \eqref{eq:iwaniecslemma} we have included the possible small eigenvalues in the error term.

We let $\ell \geq 1$ denote the width of the cusp $\infty$ with respect to $\Gamma$, so that we have $\Gamma_\infty = \{\pm \begin{psmallmatrix}
    1 & \ell \Z\\
    0 & 1
\end{psmallmatrix}\}$. In \cite[Section 9.1]{reche2026chebotarevgeodesictheoremsplit} we introduced certain functions $Z(s; n)$ in order to generalize Equation \eqref{eq:connectiontoLvalues} above. More precisely, for $|n| \geq 3$ the function $Z(s; n)$ is defined in $\text{Re}(s) > 1$ by an absolutely convergent Dirichlet series 
\begin{equation*}
    l^s Z(s; n) := \sum_{c \geq 1}\frac{a(c; n)}{c^s}
\end{equation*}
for some complex numbers $a(c; n)$, and admits a meromorphic continuation to $\CC$ with a simple pole at $s = 1$ as its only possible singularity on $\text{Re}(s) \geq 1/2$, see \cite[Lemma 9.1 and Equation (132)]{reche2026chebotarevgeodesictheoremsplit}. Equation \eqref{eq:connectiontoLvalues} can then be generalized to 
\begin{equation}\label{eq:pgtsumofresidues}
    \Psi_\Gamma(x; \chi) = \frac{\vol(\Gamma\backslash \Hyp)\pi}{2}\sum_{2 < n \leq X} (n^2 - 4)^{1/2}\text{Res}_{s = 1}Z(s; n),
\end{equation}
see \cite[Equation (134)]{reche2026chebotarevgeodesictheoremsplit}. Informally, the coefficients $a(c; n)$ are weighted counts of solutions to certain quadratic congruences modulo $cq$. Splitting these solutions by an archimedean sign one can write
\begin{equation}\label{eq:zetaintermsofsignedzetas}
    Z(s; n) = Z^+(s; n) + \chi(-I)Z^+(s; -n),
\end{equation}
where $Z^+(s; n)$ has similar properties to $Z(s; n)$,
see \cite[Section 9.4]{reche2026chebotarevgeodesictheoremsplit}. We know that if $\theta$ is any admissible exponent, then there exists $B > 0$ such that for any $\varepsilon>0$ the bound
\begin{equation}\label{eq:subconvexitybykovskii}
    Z^+(s; n) \ll_{q, \theta, \varepsilon} (1 + |s|)^B |n|^{2\theta + \varepsilon}
\end{equation}  
holds uniformly for $|n| \geq 3$ and $s$ satisfying $\text{Re}(s) \geq 1/2$ and $|s-1| \geq \varepsilon$, see \cite[Proposition 9.7]{reche2026chebotarevgeodesictheoremsplit}. Furthermore, by grouping the solutions which the coefficients $a(c; n)$ count by their reduction modulo $q$ one can derive the expression
\begin{equation}\label{eq:expressionforgeneralzagierzeta}
    Z^+(s; n) = \frac{\ell^{1-s}}{q\varphi(q)}\sum_{\substack{\widetilde{\gamma} \in \widetilde{\Gamma}\\
    \Tr(\widetilde{\gamma}) = n}}\overline{\chi(\widetilde{\gamma})}\sum_{a\in (\Z/q\Z)^\times}Z_q(s; n; \widetilde{\gamma}; a)\sum_{\psi \modulo q} \overline{\psi(a)}H(s; n; \psi)\frac{L(s; \psi)}{L(2s; \psi)}L(s; \psi_D \psi),
\end{equation}
see the proof of Proposition 9.7 in \cite{reche2026chebotarevgeodesictheoremsplit}. In \eqref{eq:expressionforgeneralzagierzeta}, $D$ is the fundamental discriminant for which we can write $n^2 - 4 =Df^2$ for some integer $f$, $\psi_D$ is the real Dirichlet character associated to $D$, and the functions $s \mapsto Z_q(s; n; \widetilde{\gamma}; a)$ and $s \mapsto  H(s; n; \psi)$ are holomorphic on $\text{Re}(s) \geq 1/2$ and satisfy 
\begin{equation}\label{eq:boundsbadprimes}
    Z_q(s; n; \widetilde{\gamma}; a), H(s; n; \psi) \ll_{q, \varepsilon} |ns|^\varepsilon
\end{equation}
for $|n| \geq 3$ and $\text{Re}(s) \geq 1/2$. 

\begin{lem}\label{lem:numberofrepeatedfundamentaldiscriminants}
For each fundamental discriminant $D > 0$ we have  
\begin{equation*}
    \#\{3 \leq n \leq X : n^2 - 4 = Df^2 \text{ for some }f\} \ll \log X,
\end{equation*}
where the implied constant is independent in $D$.
\end{lem}
\begin{proof}
Let $K$ be the field $Q(\sqrt{D})$, with ring of integers $\O_K$. We consider $K$ embedded in $\R$ so that $\sqrt{D} > 0$. By Dirichlet's unit theorem we know that $\O_K^1 = \{\pm z_K^m: m \in \Z\}$ for a unique $z_K$ satisfying $z_K > 1$. It is easy to see that we have $z_K \geq (3 + \sqrt{5})/2$. If $n \geq 3$ is such that $n^2 - 4 = Df^2$ for some integer $f$, then the element $z_n := (n + \sqrt{n^2 - 4})/2$ belongs to $\O_K^1$ and satisfies $z_n > 1$, so we can write $z_n = z_K^m$ for some $m \geq 1$. It follows that 
\begin{equation*}
    \#\{3 \leq n \leq X : n^2 - 4 = Df^2 \text{ for some }f\} \ll \frac{\log X}{\log z_K} \ll \log X,
\end{equation*}
as desired.
\end{proof}

\begin{lem}\label{lem:lindelof}
If $L(s, \psi_D\psi)$ has no zero in the rectangle $[\sigma, 1]\times [-T, T]$, then 
\begin{equation*}
    L(s, \psi_D\psi) \ll_{q, \varepsilon} (DT)^\varepsilon
\end{equation*}
for $s$ such that $\text{Re}(s) > \sigma + \varepsilon$, $|\text{Im}(s)| < T - 1$ and $|s-1| \geq \varepsilon$. 
\end{lem}
\begin{proof}
This is well-known, see for example \cite[Lemma 1]{bykovskiiprimegeodesicshortintervals}.
\end{proof}

\begin{lem}\label{lem:mostzetasatisfylindelof}
Let $1/2 \leq \sigma \leq 1$ and $x \geq 1$. For any $\varepsilon > 0$ there is a set $\B_\sigma \subset (x^{1/2}, 2x^{1/2}]$ of cardinality $|\B_\sigma| \ll_\varepsilon x^{2(1 + 2\theta)(1-\sigma) + \varepsilon}$ such that the bound 
\begin{equation}\label{eq:lindelofsigma}
    Z(s; n) \ll_\varepsilon x^\varepsilon
\end{equation}
holds uniformly for $n \in (x^{1/2}, 2x^{1/2}] \setminus \B_\sigma$ and all $s$ such that $\text{Re}(s) \geq \sigma$, $|\text{Im}(s)|\leq x^\varepsilon$ and $|s-1| \geq \varepsilon$.  
\end{lem}

\begin{proof}
The lemma follows by combining Proposition \ref{prop:myzerodensity}, identities \eqref{eq:zetaintermsofsignedzetas} and \eqref{eq:expressionforgeneralzagierzeta}, the bounds in \eqref{eq:boundsbadprimes}, and Lemmas \ref{lem:numberofrepeatedfundamentaldiscriminants} and \ref{lem:lindelof}.
\end{proof}

\section{Proof of Theorem \ref{thm:maintheorem}.}\label{sec:proofmaintheorem}
Recall the setup of quaternion algebras introduced above before the statement of Theorem \ref{thm:maintheorem}. If $B$ is a division algebra, then by \cite[Theorem 1.2]{reche2026chebotarevgeodesictheoremnonsplit} we know that there exists a class function $f$ on $\SL_2(\Z/(q\disc(\D))\Z)$ such that 
\begin{equation*}
\Psi_{\Gamma}(x; \varphi) = \Psi_{\SL_2(\Z)}(x; f) + O_{\varphi, \varepsilon}(x^{1/2 + \varepsilon}).
\end{equation*}
Thus, in Theorem \ref{thm:maintheorem} we can restrict ourselves to the case $B = M_2(\Q)$ and $\O = M_2(\Z)$. Also, the same argument as in \cite[Section 2]{reche2026chebotarevgeodesictheoremsplit} shows the following variant of \cite[Proposition 2.3]{reche2026chebotarevgeodesictheoremsplit}: proving Theorem \ref{thm:maintheorem} for $B = M_2(\Q)$ and $\O = M_2(\Z)$ is equivalent to proving that
\begin{equation}\label{eq:newexponentcongruencecharacters}
    \Psi_{\Gamma}(x; \chi) = \delta_{\chi = 1} x +  O_{q, \varepsilon}(x^{\eta(\theta) + \varepsilon})
\end{equation}
holds for any congruence character $(\Gamma, \chi)$ of any level $q$. From now on we fix a congruence character $(\Gamma, \chi)$ of level $q$ and work towards proving \eqref{eq:newexponentcongruencecharacters}. Recall that we let $k \in \{0, 1\}$ denote the weight of $\chi$.

Let $\omega \in C_c^\infty(0, 1)$ satisfy $\int_\R \omega(u) \, du = 1$, let $Y$ be a parameter satisfying $x^{1/2 + \varepsilon} \leq Y \leq x/\log x$ and define $\omega_Y(u) := \omega(u/Y)/Y$. We have the following generalization of \eqref{eq:soundyoungsmoothing}.
\begin{prop}\label{prop:smoothingsum}
We have 
\begin{equation}\label{eq:smoothingtobeproved}
    \Psi_\Gamma(x; \chi) = \delta_{\chi = 1}x - \int_0^Y (\Psi_\Gamma(x + u; \chi) - \Psi_\Gamma(x; \chi) - \delta_{\chi = 1}u)\omega_Y(u) \, du + O_\varepsilon(x^{7/8 + \varepsilon}Y^{-1/4} + \delta_{k = 0}x^{1/2 + 7/64}). 
\end{equation}
\end{prop}
\begin{proof}
Using \eqref{eq:luosarnakbound} instead of \cite[Equation (58)]{luosarnak}, and using \eqref{eq:iwaniecslemma} instead of \cite[Lemma 1]{Iwaniec1984}, the proof of Proposition \ref{prop:smoothingsum} is exactly the same as the proof of Equation (13) from \cite{Soundararajan_2013}.
\end{proof}

The following proposition generalizes the discussion of Section \ref{subsec:mainidea}.

\begin{prop}\label{prop:mylocalinterval}
Let $\theta$ be an admissible exponent. For $1 \leq V \leq x$, $1/2 \leq \sigma \leq 1$ and $x^{1/2 + \varepsilon} \leq u \leq x^{1-\varepsilon}$ we have 
\begin{equation*}
    \Psi_\Gamma(x + u; \chi) - \Psi_\Gamma(x; \chi) = \delta_{\chi = 1}u + O_\varepsilon\left(x^{1/2 + \varepsilon}\left(x^{2(1 + 2\theta)(1-\sigma) + \theta}V^{-1/2} + V^{1/2} + ux^{-1/2}V^{\sigma - 1}\right)\right).
\end{equation*}
\end{prop}
\begin{proof}
Let $\omega \in C_c^\infty(0, \infty)$ satisfy $\int_\R \omega(x) \, dx = 1$. Let $\widehat{\omega}$ denote the Mellin transform of $\omega$. Note that $|\widehat{\omega}(s)| \ll_B (1 + |s|)^{-B}$ for arbitrary $B > 0$ on any fixed vertical strip. Consider 
\begin{equation}\label{eq:approximationtogeometricsideshortintervals}
    \S_V := \frac{\vol(\Gamma\backslash \Hyp)\pi}{2V}\sum_{c\geq 1}\sum_{X < n \leq X'} (n^2 - 4)^{1/2} a(c; n)\omega(c \ell/V)
\end{equation}
with $X = \sqrt{x} + 1/\sqrt{x}$ and $X' = \sqrt{x + u} + 1/\sqrt{x + u}$.
We know from \cite[Equation (146)]{reche2026chebotarevgeodesictheoremsplit} that 
\begin{equation}\label{eq:auxevaluationapprox}
    \S_V = \delta_{\chi = 1}u + O_{q, \varepsilon}(V^{1/2 + \varepsilon}x^{1/2} + uV^{-1}).
\end{equation}
If we execute first the summation over $c$ in \eqref{eq:approximationtogeometricsideshortintervals}, apply Mellin inversion formula, move the contour to the left, pick up the residues at $s = 1$ and use \eqref{eq:pgtsumofresidues} we obtain
\begin{equation}\label{eq:auxeeqaftermellin}
    \S_V = \Psi_\Gamma(x + u; \chi) - \Psi_\Gamma(x; \chi) + \frac{\vol(\Gamma\backslash \Hyp)\pi}{2} \sum_{X < n \leq X'}(n^2 - 4)^{1/2}I_V(n^2 - 4),
\end{equation}
where 
\begin{equation*}
    I_V(n^2 - 4) := \frac{1}{2\pi i}\int_{(\eta)} Z(s; n)\widehat{\omega}(s)V^{s-1}\, ds
\end{equation*}
with $1/2 \leq \eta < 1$. We bound the sum of integrals by applying Lemma \ref{lem:mostzetasatisfylindelof} and distinguishing two cases:
\begin{enumerate}
    \item [i)] If $n \in \B_\sigma$, we let $\eta = 1/2$ and use \eqref{eq:subconvexitybykovskii} to obtain 
    \begin{equation*}
        I_V(n^2 - 4) \ll_\varepsilon x^{\theta + \varepsilon} V^{-1/2}.
    \end{equation*}
    \item [ii)] If $n \in (X, X']\setminus \B_\sigma$ we let $\eta = \sigma + \varepsilon$ and use the bound in \eqref{eq:lindelofsigma} and the rapid decay of $\widehat{\omega}(s)$ to deduce 
    \begin{equation*}
        I_V(n^2 - 4) \ll_\varepsilon x^{\varepsilon} V^{\sigma - 1}.
    \end{equation*}
\end{enumerate}
We know that $|\B_\sigma| \ll x^{2(1+2\theta)(1-\sigma) + \varepsilon}$ and $X' - X \asymp ux^{-1/2}$. Summing over $n$ we get 
\begin{equation}\label{eq:auxeqboundsumintegrals}
    \sum_{X < n \leq X'}(n^2 - 4)^{1/2}I_V(n^2 - 4) \ll_\varepsilon x^{1/2 + \varepsilon}\left( x^{2(1 + 2\theta)(1-\sigma) + \theta} V^{-1/2} + ux^{-1/2}V^{\sigma - 1}\right). 
\end{equation}
The proof of Proposition \ref{prop:mylocalinterval} is concluded by combining \eqref{eq:auxevaluationapprox}, \eqref{eq:auxeeqaftermellin} and \eqref{eq:auxeqboundsumintegrals}.
\end{proof}

Combining Proposition \ref{prop:smoothingsum} and Proposition \ref{prop:mylocalinterval} we obtain 
\begin{equation*}
    \Psi_\Gamma(x; \chi) = \delta_{\chi = 1}x + O_\varepsilon\left(x^{1/2 +\varepsilon}\left(x^{3/8}Y^{-1/4} + x^{2(1 + 2\theta)(1-\sigma) + \theta}V^{-1/2} + V^{1/2} + Yx^{-1/2}V^{\sigma - 1} + \delta_{k = 0}x^{7/64}\right)\right),
\end{equation*}
where $1/2 \leq \sigma \leq 1$, $1 \leq V \leq x$ and $x^{1/2 + \varepsilon} \leq Y \leq x^{1 - \varepsilon}$. It is a routine calculation to check that the error term is optimized with the choice 
\begin{equation*}
\begin{aligned}
    1 - \sigma & = \frac{-5-11\theta + \sqrt{81\theta^2 + 106\theta +33}}{4(1 + 2\theta)},\\
    V & = x^{2(1 + 2\theta)(1-\sigma) + \theta},\\
    Y & = x^{7/10 + 8(1 + 2\theta)(1-\sigma)^2/5 + 4\theta(1-\sigma)/5},
\end{aligned}
\end{equation*}
which leads to \eqref{eq:newexponentcongruencecharacters} and completes the proof of Theorem \ref{thm:maintheorem}.

\newpage

\newgeometry{a4paper, total={7in, 9.4in}}

\appendix
\section*{Appendix}

The first version of this paper was submitted before OpenAI's release, on 6 October 2026, of a manuscript claiming a proof that Dirichlet $L$-functions have no zeros on the region $\text{Re}(s) > 7/8$. A follow-up preprint claims a slightly larger zero-free region \cite{liu2026slightlyimprovedzerofreehalfplanes}. If Dirichlet $L$-functions had no zeros on the region $\text{Re}(s) > \sigma_0$ then, since $\B_{\sigma_0} = \emptyset$, Proposition \ref{prop:mylocalinterval} would hold for $\sigma = \sigma_0$ with the first term in the error omitted. After optimization this would lead to the exponent $\eta'(\sigma_0) = 1/2 + 1/(7 - 2\sigma_0)$ in Theorem \ref{thm:maintheorem}. In particular, for $\sigma_0 = 7/8$ one would have $\eta'(7/8) = 29/42 \approx 0.690476 < 0.69359 \approx \eta(1/6)$, improving on Theorem \ref{thm:maintheorem}. 

\printbibliography

@article {heathbrownquadlargesieve,
    AUTHOR = {Heath-Brown, D.},
     TITLE = {A mean value estimate for real character sums},
  JOURNAL = {Acta Arithmetica},
    VOLUME = {72},
      YEAR = {1995},
    NUMBER = {3},
     PAGES = {235--275},
}

@misc{liu2026slightlyimprovedzerofreehalfplanes,
	title={Slightly improved zero-free half-planes for the quasi-Riemann hypothesis}, 
	author={Liu, B.},
	year={2026},
	eprint={2610.12234},
	archivePrefix={arXiv},
	primaryClass={math.NT},
}

@misc{reche2026chebotarevgeodesictheoremsplit,
      title={Chebotarev geodesic theorem: split case}, 
      author={Alberto Acosta Reche},
      year={2026},
      eprint={2606.25903},
      archivePrefix={arXiv},
      primaryClass={math.NT},
}

@misc{reche2026chebotarevgeodesictheoremnonsplit,
      title={Chebotarev geodesic theorem: non-split case}, 
      author={Alberto Acosta Reche},
      year={2026},
      eprint={2608.04923},
      archivePrefix={arXiv},
      primaryClass={math.NT},
}

@book{montgomerytopics,
    AUTHOR = {Montgomery, H. L.},
     TITLE = {Topics in multiplicative number theory},
    SERIES = {Lecture Notes in Mathematics},
    VOLUME = {Vol. 227},
 PUBLISHER = {Springer-Verlag, Berlin-New York},
      YEAR = {1971},
     PAGES = {ix+178},
}

@article {bykovskiiprimegeodesicshortintervals,
    AUTHOR = {Bykovski\u i, V. A.},
     TITLE = {Density theorems and the mean value of arithmetic functions on
              short intervals},
  JOURNAL = {Rossi\u iskaya Akademiya Nauk. Sankt-Peterburgskoe Otdelenie.
              Matematicheski\u i\ Institut im. V. A. Steklova. Zapiski
              Nauchnykh Seminarov (POMI)},
    VOLUME = {212},
      YEAR = {1994},
     PAGES = {56--70, 196},
}

@article {petrowyoung20,
    AUTHOR = {Petrow, I. and Young, M. P.},
     TITLE = {The {W}eyl bound for {D}irichlet {$L$}-functions of cube-free
              conductor},
  JOURNAL = {Annals of Mathematics},
    VOLUME = {192},
      YEAR = {2020},
    NUMBER = {2},
     PAGES = {437--486},
}

@article {petrowyoung,
    AUTHOR = {Petrow, I. and Young, M. P.},
     TITLE = {The fourth moment of {D}irichlet {$L$}-functions along a coset
              and the {W}eyl bound},
  JOURNAL = {Duke Mathematical Journal},
    VOLUME = {172},
      YEAR = {2023},
    NUMBER = {10},
     PAGES = {1879--1960},
}

@article{luosarnak,
author = {Luo, W. and Sarnak, P.},
title = {{Quantum Ergodicity of Eigenfunctions on $PSL_2(\mathbf{Z})/\mathbf{H}^2$}},
volume = {81},
journal = {Publications Mathématiques de l'IHÉS},
pages = {207 -- 237},
year = {1995}
}

@article{Iwaniec1984,
    AUTHOR = {Iwaniec, H.},
     TITLE = {Prime geodesic theorem},
  JOURNAL = {Journal f\"ur die Reine und Angewandte Mathematik. [Crelle's
              Journal]},
    VOLUME = {349},
      YEAR = {1984},
     PAGES = {136--159},
}

@article{DFIsubconvexityartin,
    AUTHOR = {Duke, W. and Friedlander, J. B. and Iwaniec, H.},
     TITLE = {The subconvexity problem for {A}rtin {$L$}-functions},
  JOURNAL = {Inventiones Mathematicae},
    VOLUME = {149},
      YEAR = {2002},
    NUMBER = {3},
     PAGES = {489--577},
}

@article{conreyiwaniec,
    AUTHOR = {Conrey, J. B. and Iwaniec, H.},
     TITLE = {The cubic moment of central values of automorphic
              {$L$}-functions},
  JOURNAL = {Annals of Mathematics. Second Series},
    VOLUME = {151},
      YEAR = {2000},
    NUMBER = {3},
     PAGES = {1175--1216},
}

@article{kimramakrishnansarnak2003,
AUTHOR = "Kim, H.",
     TITLE = {Functoriality for the exterior square of {${\rm GL}_4$} and
              the symmetric fourth of {${\rm GL}_2$}},
      NOTE = {With appendix 1 by Dinakar Ramakrishnan and appendix 2 by Kim
              and Peter Sarnak},
  JOURNAL = {Journal of the American Mathematical Society},
    VOLUME = {16},
      YEAR = {2003},
    NUMBER = {1},
     PAGES = {139--183},
}

@article{Soundararajan_2013,
    AUTHOR = {Soundararajan, K. and Young, M. P.},
     TITLE = {The prime geodesic theorem},
  JOURNAL = {Journal f\"ur die Reine und Angewandte Mathematik. [Crelle's
              Journal]},
    VOLUME = {676},
      YEAR = {2013},
     PAGES = {105--120},
}

@article {Cai2002,
    AUTHOR = {Cai, Y.},
     TITLE = {Prime geodesic theorem},
JOURNAL = {Journal de Th\'eorie des Nombres de Bordeaux},
    VOLUME = {14},
      YEAR = {2002},
    NUMBER = {1},
     PAGES = {59--72},
}

@book{iwanieckowalski,
    AUTHOR = {Iwaniec, H. and Kowalski, E.},
     TITLE = {Analytic number theory},
    SERIES = {American Mathematical Society Colloquium Publications},
    VOLUME = {53},
 PUBLISHER = {American Mathematical Society, Providence, RI},
      YEAR = {2004},
     PAGES = {xii+615},
}

@book{iwaniec,
    AUTHOR = {Iwaniec, H.},
     TITLE = {Spectral methods of automorphic forms},
    SERIES = {Graduate Studies in Mathematics},
    VOLUME = {53},
   EDITION = {Second},
 PUBLISHER = {American Mathematical Society, Providence, RI; Revista
              Matem\'atica Iberoamericana, Madrid},
      YEAR = {2002},
     PAGES = {xii+220},
}

\footnotesize
\textit{Email address}: \, \texttt{alberto.reche.23@ucl.ac.uk}\par\nopagebreak
\textsc{Department of Mathematics, University College London, 25 Gordon Street, London WC1H 0AY, United Kingdom}

\end{document}